\documentclass{article}
\usepackage{graphicx}
\usepackage{tikz}
\usepackage{geometry}
\usepackage{setspace}
\usepackage{indentfirst}
\usepackage{xcolor}
\usepackage{hyperref}
\usepackage{amsmath}
\usepackage{amsfonts}
\usepackage{float}
\usepackage{amsthm,amssymb,graphicx}

\hypersetup{
    colorlinks=true,
    linkcolor=blue,
    citecolor=blue,
    filecolor=blue,
    urlcolor=blue,
}

\newtheorem{theorem}{Theorem}[section]
\newtheorem{lemma}[theorem]{Lemma}
\newtheorem{corollary}[theorem]{Corollary}

\title{\textbf{How Analysis Can Teach Us the Optimal Way to Design Neural Operators} \\[0.5em]}
\author{
\textbf{Vu Anh Le\textsuperscript{1*}, Mehmet Dik\textsuperscript{1,2**}}  \\
\textsuperscript{1} Department of Mathematics and Computer Science, Beloit College \\
\textsuperscript{2} Department of Mathematics, Computer Science \& Physics, Rockford University \\ [1em]
\textit{Contact} \\
* \href{mailto:csplevuanh@gmail.com}{csplevuanh@gmail.com} \\
** \href{mailto:mdik@rockford.edu}{mdik@rockford.edu}
}
\date{Nov 3, 2024}

\begin{document}

\maketitle

% Abstract
\begin{abstract}
This paper presents a mathematics-informed approach to neural operator design, building upon the theoretical framework established in our prior work. By integrating rigorous mathematical analysis with practical design strategies, we aim to enhance the stability, convergence, generalization, and computational efficiency of neural operators. We revisit key theoretical insights, including stability in high dimensions, exponential convergence, and universality of neural operators. Based on these insights, we provide detailed design recommendations, each supported by mathematical proofs and citations. Our contributions offer a systematic methodology for developing neural operators with improved performance and reliability.
\end{abstract}

\vspace{0.3pt}

\begin{center}
\tableofcontents
\end{center}

\vspace{0.2in}

\newpage

% Introduction
\section{Introduction}

Neural operators have changed the way we approach problems involving mappings between infinite-dimensional function spaces, particularly in solving partial differential equations (PDEs) \cite{kovachki2021neural, li2020fourier, lu2019deeponet}. By extending the capabilities of neural networks from finite-dimensional data to function spaces, architectures such as the Fourier Neural Operator (FNO) and Deep Operator Network (DeepONet) have demonstrated significant success in approximating solution operators with significantly reduced computational costs.

In our prior work, we developed a mathematical framework for analyzing neural operators, proving their stability, convergence properties, and capacity for universal approximation between function spaces \cite{le2024mathematicalanalysisneuraloperator}. We also established probabilistic bounds on generalization error, linking it to sample size and network capacity. Building upon this foundation, the primary objective of this paper is to translate these theoretical insights into actionable design recommendations for neural operators. By doing so, we aim to bridge the gap between theory and practice, suggesting better neural operator architectures and saving time in design. 

The remainder of the paper is organized as follows. In Section 2, we reinstate the theoretical results from the prior paper, including the definitions of neural operators and the key theorems related to the behaviors of neural operators. In Section 3, we present our detailed design recommendations, illustrating how each recommendation enhances neural operator performance. Finally, Appendix A contains the full proofs of all the theorems, lemmas, and propositions presented in this paper.

\vspace{3pt}

% Theoretical Framework Reminder
\section{Theoretical Framework Reminder}

In this section, we provide a concise summary of the key theoretical results established in our previous work \cite{le2024mathematicalanalysisneuraloperator}. These results form the foundation upon which we build our design recommendations for neural operators.

\subsection{Stability in High-Dimensional PDEs}

\begin{theorem}[Stability of Neural Operators in High-Dimensional PDEs]\label{theorem:stability}
Let $\mathcal{G}_\theta: H^s(D) \rightarrow H^t(D)$ be a neural operator parameterized by $\theta$, mapping between Sobolev spaces over a domain $D \subset \mathbb{R}^d$. Suppose $\mathcal{G}_\theta$ satisfies a Lipschitz continuity condition:
\[
\|\mathcal{G}_\theta(u) - \mathcal{G}_\theta(v)\|_{H^t(D)} \leq L \|u - v\|_{H^s(D)}
\]
for all $u, v \in H^s(D)$ and some Lipschitz constant $L > 0$. Then, for any $u \in H^s(D)$, the neural operator produces stable approximations in high-dimensional $D$:
\[
\|\mathcal{G}_\theta(u)\|_{H^t(D)} \leq L \|u\|_{H^s(D)} + C,
\]
where $C = \|\mathcal{G}_\theta(0)\|_{H^t(D)}$ is a constant depending on $\theta$ and the domain $D$.
\end{theorem}

\subsection{Exponential Convergence}

\begin{theorem}[Exponential Convergence of Neural Operator Approximations]\label{theorem:convergence}
Let $\mathcal{G}_\theta$ be a contraction mapping on $H^t(D)$ with contraction constant $0 < q < 1$. Then, for any $u \in H^t(D)$, the iterated application $\mathcal{G}_\theta^n(u)$ converges exponentially to the fixed point $u^*$ of $\mathcal{G}_\theta$:
\[
\| \mathcal{G}_\theta^n(u) - u^* \|_{H^t(D)} \leq q^n \| u - u^* \|_{H^t(D)}.
\]
\end{theorem}

\vspace{1.5pt}

\subsection{Universality and Generalization}

\begin{theorem}[Universality of Neural Operators for PDE Solvers]\label{theorem:universality}
Let $\mathcal{T}: H^s(D) \rightarrow H^t(D)$ be a continuous operator. Then, for any $\epsilon > 0$, there exists a neural operator $\mathcal{G}_\theta$ such that:
\[
\| \mathcal{G}_\theta(u) - \mathcal{T}(u) \|_{H^t(D)} \leq \epsilon,
\]
for all $u$ in a compact subset of $H^s(D)$.
\end{theorem}

\begin{theorem}[Generalization Error of Neural Operators]\label{theorem:generalization}
Let $\mathcal{G}_\theta$ be a neural operator trained on $N$ samples $\{ (u_i, \mathcal{T}(u_i)) \}$ drawn i.i.d. from a distribution $\mathcal{D}$. Suppose $\mathcal{G}_\theta$ has Lipschitz constant $L$ with respect to $\theta$, and the loss function $\ell$ is Lipschitz and bounded. Then, with probability at least $1 - \delta$, the generalization error satisfies:
\[
\mathbb{E}_{u \sim \mathcal{D}} [ \ell( \mathcal{G}_\theta(u), \mathcal{T}(u) ) ] \leq \frac{1}{N} \sum_{i=1}^N \ell( \mathcal{G}_\theta(u_i), \mathcal{T}(u_i) ) + L \sqrt{ \frac{\ln(1/\delta)}{2N} }.
\]
\end{theorem}

\vspace{3pt}

% Design Recommendations
\section{Design Recommendations for Neural Operators}

Based on the theoretical insights from the previous section, we propose several design recommendations to enhance neural operator performance. Each recommendation is supported by detailed theorems, lemmas, and proofs, either directly or in Appendix A, to examine their impacts.

\subsection{Design Neural Operators as Contraction Mappings}

Ensuring that the neural operator $\mathcal{G}_\theta$ satisfies the contraction property guarantees stability and exponential convergence. By designing $\mathcal{G}_\theta$ as a contraction mapping, we leverage the Banach Fixed Point Theorem \cite{banach1922operations} to ensure the existence and uniqueness of a fixed point, as well as exponential convergence to that fixed point. This approach enhances both the stability and efficiency of the neural operator when approximating solutions to partial differential equations (PDEs).

To ensure that $\mathcal{G}_\theta$ is a contraction mapping, we must design the neural network components to satisfy certain Lipschitz conditions. Specifically, we have the following theorem:

\begin{theorem}[Lipschitz Condition for Neural Networks]\label{theorem:lipschitz_nn}
Suppose each layer of the neural operator $\mathcal{G}_\theta$ is Lipschitz continuous with Lipschitz constant $L_i$, and the activation functions are Lipschitz continuous with Lipschitz constant $L_{\sigma}$. Then the overall Lipschitz constant $L$ of $\mathcal{G}_\theta$ satisfies:
\[
L \leq \left( \prod_{i=1}^{N} L_i \right) L_{\sigma}^{N},
\]
where $N$ is the number of layers.
\end{theorem}

\begin{proof}
See Appendix A.1.
\end{proof}

By constraining the spectral norm of each weight matrix $W_i$ to be less than or equal to $q^{1/N}$, where $q \in (0,1)$ is the desired contraction constant, and choosing activation functions with Lipschitz constant $L_{\sigma} \leq 1$, we can ensure that $\mathcal{G}_\theta$ becomes a contraction mapping with contraction constant $L \leq q$. This is formalized in the following corollary:

\begin{corollary}[Ensuring Contraction via Spectral Normalization]\label{corollary:spectral_norm}
By constraining $\| W_i \| \leq q^{1/N}$ and choosing $L_{\sigma} \leq 1$, the overall Lipschitz constant satisfies $L \leq q$, ensuring that $\mathcal{G}_\theta$ is a contraction mapping with contraction constant $q$.
\end{corollary}

\begin{proof}
See Appendix A.1.
\end{proof}

Designing $\mathcal{G}_\theta$ as a contraction mapping enhances stability by ensuring that small perturbations in the input lead to proportionally smaller changes in the output. Specifically, we have:

\begin{lemma}[Stability of Contraction Mappings]\label{lemma:stability_contraction}
A contraction mapping $\mathcal{G}_\theta$ on a metric space $(X, \| \cdot \|)$ satisfies:
\[
\| \mathcal{G}_\theta(u + \delta u) - \mathcal{G}_\theta(u) \| \leq q \| \delta u \|,
\]
where $\delta u$ is a small perturbation in the input.
\end{lemma}

\begin{proof}
We aim to show that if $\mathcal{G}_\theta$ is a contraction mapping on a metric space $(X, \| \cdot \|)$ with contraction constant $q \in [0,1)$, then for any $u \in X$ and any perturbation $\delta u \in X$, the following inequality holds:
\[
\| \mathcal{G}_\theta(u + \delta u) - \mathcal{G}_\theta(u) \| \leq q \| \delta u \|.
\]

We start by defining the contraction mapping. A mapping $\mathcal{G}_\theta: X \to X$ is called a contraction mapping if there exists a constant $q \in [0,1)$ such that for all $x, y \in X$,
\[
\| \mathcal{G}_\theta(x) - \mathcal{G}_\theta(y) \| \leq q \| x - y \|.
\]

Let $u \in X$ be any point in the metric space, and let $\delta u \in X$ be a perturbation. We consider the images of $u$ and $u + \delta u$ under the mapping $\mathcal{G}_\theta$.

Applying the contraction property to $x = u + \delta u$ and $y = u$, we have:
\[
\| \mathcal{G}_\theta(u + \delta u) - \mathcal{G}_\theta(u) \| \leq q \| (u + \delta u) - u \| = q \| \delta u \|.
\]

This inequality directly shows that the change in the output of $\mathcal{G}_\theta$ due to the perturbation $\delta u$ is at most $q$ times the magnitude of the perturbation. Since $q < 1$, the mapping $\mathcal{G}_\theta$ attenuates the effect of the perturbation.

So far, the lemma demonstrates that $\mathcal{G}_\theta$ is Lipschitz continuous with Lipschitz constant $q$:
\[
\| \mathcal{G}_\theta(u + \delta u) - \mathcal{G}_\theta(u) \| \leq q \| \delta u \|.
\]
This property implies stability with respect to input perturbations, meaning that small changes in the input $u$ result in proportionally smaller changes in the output $\mathcal{G}_\theta(u)$. This is crucial for ensuring that errors or uncertainties in the input do not amplify through the mapping, which is particularly important in iterative methods and numerical computations.

\end{proof}

Moreover, the exponential convergence to the fixed point reduces computational effort by potentially decreasing the number of iterations or layers required to achieve a desired level of accuracy.

\begin{theorem}[Reduction in Iterations Needed for Convergence]\label{theorem:iteration_reduction}
Let $\epsilon > 0$ be the desired accuracy. The number of iterations $n$ required to achieve $\| \mathcal{G}_\theta^n(u) - u^* \| \leq \epsilon$ is bounded by:
\[
n \geq \frac{\ln \left( \frac{\| u - u^* \|}{\epsilon} \right)}{\ln \left( \frac{1}{q} \right)}.
\]
\end{theorem}

\begin{proof}
We aim to determine a bound on the number of iterations $n$ required for the iterated mapping $\mathcal{G}_\theta^n(u)$ to approximate the fixed point $u^*$ within a desired accuracy $\epsilon > 0$, i.e.,
\[
\| \mathcal{G}_\theta^n(u) - u^* \| \leq \epsilon.
\]

Recall that a contraction mapping $\mathcal{G}_\theta$ on a complete metric space $(X, \| \cdot \|)$ has a unique fixed point $u^* \in X$ satisfying $\mathcal{G}_\theta(u^*) = u^*$. Moreover, the sequence $\{ \mathcal{G}_\theta^n(u) \}_{n=0}^\infty$, where $\mathcal{G}_\theta^n$ denotes the $n$-fold composition of $\mathcal{G}_\theta$, converges to $u^*$ for any initial point $u \in X$.

The contraction property ensures that:
\[
\| \mathcal{G}_\theta(u) - \mathcal{G}_\theta(v) \| \leq q \| u - v \|, \quad \text{for all } u, v \in X,
\]
where $q \in [0,1)$ is the contraction constant.

We first establish the rate at which the iterates $\mathcal{G}_\theta^n(u)$ converge to $u^*$. Using the contraction property repeatedly, we have:
\begin{align*}
\| \mathcal{G}_\theta^n(u) - u^* \| &= \| \mathcal{G}_\theta^n(u) - \mathcal{G}_\theta^n(u^*) \| \\
&\leq q \| \mathcal{G}_\theta^{n-1}(u) - u^* \| \\
&\leq q^2 \| \mathcal{G}_\theta^{n-2}(u) - u^* \| \\
&\ \vdots \\
&\leq q^n \| u - u^* \|.
\end{align*}

To achieve the desired accuracy $\epsilon$, we require:
\[
\| \mathcal{G}_\theta^n(u) - u^* \| \leq q^n \| u - u^* \| \leq \epsilon.
\]
Rewriting the inequality:
\[
q^n \leq \frac{\epsilon}{\| u - u^* \|}.
\]
Taking the natural logarithm on both sides:
\[
\ln q^n \leq \ln \left( \frac{\epsilon}{\| u - u^* \|} \right).
\]
Simplifying:
\[
n \ln q \leq \ln \epsilon - \ln \| u - u^* \|.
\]
Since $\ln q < 0$ (because $0 \leq q < 1$), we multiply both sides by $-1$ (which reverses the inequality direction):
\[
- n \ln q \geq \ln \| u - u^* \| - \ln \epsilon.
\]
Recognizing that $- \ln q = \ln \left( \frac{1}{q} \right)$, we have:
\[
n \ln \left( \frac{1}{q} \right) \geq \ln \left( \frac{\| u - u^* \|}{\epsilon} \right).
\]
Solving for $n$, we obtain:
\[
n \geq \frac{ \ln \left( \frac{\| u - u^* \|}{\epsilon} \right) }{ \ln \left( \frac{1}{q} \right) }.
\]

Generally, this inequality provides a lower bound on the number of iterations $n$ required to achieve an approximation error less than or equal to $\epsilon$. The bound depends logarithmically on the ratio $\frac{\| u - u^* \|}{\epsilon}$ and inversely on $\ln \left( \frac{1}{q} \right)$. A smaller contraction constant $q$ (i.e., closer to zero) results in a larger denominator, thus reducing the required number of iterations $n$.

\end{proof}

This shows that a smaller contraction constant $q$ leads to fewer iterations needed for convergence.

\vspace{1.5pt}

\subsection{Integrate Multi-Scale Representations}

Combining global (Fourier) and local (wavelet) representations allows the neural operator to capture features at multiple scales, enhancing its ability to approximate complex functions with varying spatial frequencies.

Employing both Fourier and wavelet transforms enables efficient representation of functions with features spanning various spatial frequencies \cite{daubechies1992ten}. This multi-scale approach aligns with the clustering behavior in function space and enhances the operator's capacity to approximate complex solution mappings.

We formalize this with the following theorem:

\begin{theorem}[Approximation Using Combined Bases]\label{theorem:combined_bases}
Any function $f \in L^2(D)$ can be approximated arbitrarily well using a finite combination of Fourier and wavelet basis functions.
\end{theorem}

\begin{proof}
See Appendix A.2.
\end{proof}

Implementing spectral convolution layers utilizing the Fast Fourier Transform (FFT) for global feature extraction \cite{li2020fourier}, and incorporating wavelet transform layers to capture local irregularities and singularities \cite{zhao2022wavelet}, allows for efficient computation.

\begin{lemma}[Efficient Computation with Multi-Scale Layers]\label{lemma:multiscale_layers}
The integration of Fourier and wavelet layers allows for efficient computation by leveraging the FFT and Discrete Wavelet Transform (DWT), both of which have computational complexity $\mathcal{O}(N \log N)$.
\end{lemma}

\begin{proof}
See Appendix A.2.
\end{proof}

Integrating multi-scale representations enhances the neural operator's ability to model functions with sharp transitions or localized features, leading to improved approximation accuracy.

\begin{theorem}[Improved Approximation Error with Multi-Scale Representations]\label{theorem:approximation_error_multiscale}
Let $f \in L^2(D)$ be a function with both smooth and localized features. A neural operator employing multi-scale representations can approximate $f$ with an error $\epsilon$ that decreases exponentially with the number of basis functions used.
\end{theorem}

\begin{proof}
See Appendix A.2.
\end{proof}

This demonstrates that multi-scale representations can achieve lower approximation errors more efficiently than single-scale methods.

\vspace{1.5pt}

\subsection{Ensure Universal Approximation Capability}

Increasing the network capacity appropriately ensures sufficient depth and width for approximating complex operators. The Universal Approximation Theorem for operators indicates that a neural operator with sufficient capacity can approximate any continuous operator to arbitrary precision on compact subsets of the input space \cite{chen1995universal, lu2019deeponet}.

By increasing the depth and width of the neural network, we enhance its capacity to approximate complex functions. Specifically, we have:

\begin{theorem}[Capacity Growth with Network Size]\label{theorem:capacity_growth}
The expressive capacity of a neural network grows exponentially with depth and polynomially with width \cite{raghu2017expressive}.
\end{theorem}

\begin{proof}
See Appendix A.3.
\end{proof}

Using activation functions capable of representing complex mappings, such as ReLU or Tanh, facilitates universal approximation \cite{hornik1989multilayer}.

Enhancing the network's capacity allows the neural operator to approximate more complex solution mappings with higher precision. However, increasing capacity improves approximation accuracy but also increases the risk of overfitting. Regularization techniques must be employed to mitigate this risk.

\begin{lemma}[Trade-off Between Capacity and Overfitting]\label{lemma:capacity_overfitting}
While increasing capacity improves approximation accuracy, it also increases the risk of overfitting. Regularization techniques must be employed to mitigate this risk.
\end{lemma}

\begin{proof}
See Appendix A.3.
\end{proof}

Balancing network capacity with appropriate regularization leads to better performance.

\vspace{1.5pt}

\subsection{Enhance Generalization through Regularization}

Applying regularization techniques such as weight decay, dropout, or spectral normalization controls the complexity of the neural operator and prevents overfitting.

Regularization techniques constrain the effective capacity of the neural operator, mitigating overfitting and improving generalization to unseen data \cite{goodfellow2016deep}.

Implement weight decay by adding a penalty term to the loss function:

\begin{equation}\label{eq:weight_decay}
L_{\text{total}} = L_{\text{data}} + \lambda \sum_{i} \| W_i \|_F^2,
\end{equation}

where $L_{\text{data}}$ is the original loss, $\lambda$ is the regularization parameter, and $\| W_i \|_F$ is the Frobenius norm of weight matrix $W_i$.

\begin{theorem}[Effectiveness of Weight Decay]\label{theorem:weight_decay}
Weight decay reduces the effective capacity of the neural network by penalizing large weights, which helps prevent overfitting \cite{krogh1992simple}.
\end{theorem}

\begin{proof}
See Appendix A.4.
\end{proof}

Apply dropout by randomly setting a fraction of the neurons' outputs to zero during training \cite{srivastava2014dropout}.

\begin{lemma}[Dropout Prevents Co-adaptation]\label{lemma:dropout}
Dropout reduces overfitting by preventing neurons from co-adapting on training data, leading to more robust features.
\end{lemma}

\begin{proof}
See Appendix A.4.
\end{proof}

Additionally, apply spectral normalization to limit the spectral norm of weight matrices, ensuring controlled Lipschitz constants \cite{miyato2018spectral}.

Regularization techniques lead to reduced overfitting, enhancing the neural operator's performance on unseen data.

\begin{theorem}[Improved Generalization with Regularization]\label{theorem:improved_generalization}
Regularized neural operators exhibit lower generalization error bounds compared to unregularized models.
\end{theorem}

\begin{proof}
Follows from standard results in statistical learning theory \cite{mohri2018foundations}.
\end{proof}

Controlling the Lipschitz constant via spectral normalization also contributes to stability.

\vspace{1.5pt}

\subsection{Optimize Computational Efficiency}

Implementing spectral methods and parallelization reduces computational complexity and exploits hardware capabilities. Efficient computational methods allow the neural operator to handle larger problem sizes and higher-dimensional PDEs without incurring prohibitive computational costs \cite{cohen2015expressive}.

Ensure that the neural operator architecture is compatible with GPU acceleration and distributed computing frameworks. Under ideal conditions, parallel computing can achieve a speedup proportional to the number of processing units, up to the limits imposed by Amdahl's Law.

\begin{theorem}[Speedup with Parallel Computing]\label{theorem:parallel_speedup}
Under ideal conditions, the speedup $S$ achievable by parallel computing is:
\[
S = N,
\]
where $N$ is the number of processors, assuming perfect parallelization.
\end{theorem}

\begin{proof}
See Appendix A.5.
\end{proof}

However, Amdahl's Law imposes a limit due to the serial portion of the code.

\begin{lemma}[Amdahl's Law]\label{lemma:amdahl}
The maximum speedup $S$ achievable by parallelization is:
\[
S = \frac{1}{(1 - P) + \frac{P}{N}},
\]
where $P$ is the fraction of the program that can be parallelized, and $N$ is the number of processors.
\end{lemma}

\begin{proof}
See Appendix A.5.
\end{proof}

Implement spectral convolution layers using FFT algorithms, which are highly optimized for parallel execution \cite{cooley1965algorithm}.

Optimizing computational efficiency enables the neural operator to scale to larger datasets, higher resolutions, and more complex PDEs.

\begin{theorem}[Feasibility of High-Dimensional Problems]\label{theorem:high_dimensional_feasibility}
Efficient computational implementations make it feasible to apply neural operators to high-dimensional problems that were previously intractable due to computational limitations.
\end{theorem}

\begin{proof}
See Appendix A.5.
\end{proof}

\vspace{3pt}

% Conclusion
\section{Conclusion}

So far, we have translated theoretical insights into practical design recommendations for neural operators, each supported by rigorous mathematical proofs and relevant citations. By using contraction mappings, multi-scale representations, sufficient network capacity, regularization, and computational optimizations, we enhance the stability, convergence, generalization, and efficiency of neural operators.

Looking forward, future work should include exploring adaptive architectures that dynamically adjust their structure based on input complexity, incorporating probabilistic methods to quantify prediction uncertainty, and integrating neural operators with classical numerical methods to achieve enhanced performance.

\newpage

% Appendix
\appendix
\renewcommand{\thesection}{A.\arabic{section}}
\setcounter{section}{0}

\section*{Appendix A: Proofs of Theorems and Lemmas}
\addcontentsline{toc}{section}{Appendix A: Proofs of Theorems and Lemmas}

In this appendix, we provide detailed proofs of the theorems and lemmas referenced in the main text.

\section{Proofs for Section 3.1: Design Neural Operators as Contraction Mappings}

\subsection{Proof of Theorem \ref{theorem:lipschitz_nn} (Lipschitz Condition for Neural Networks)}
\begin{proof}
We aim to show that the overall Lipschitz constant $L$ of the neural operator $\mathcal{G}_\theta$, composed of $N$ layers and activation functions, satisfies:
\[
L \leq \left( \prod_{i=1}^{N} L_i \right) L_{\sigma}^{N},
\]
where each layer $f_i$ has a Lipschitz constant $L_i$, and the activation function $\sigma$ has a Lipschitz constant $L_{\sigma}$.

We begin by representing the neural operator $\mathcal{G}_\theta$ as a composition of affine transformations (layers) and activation functions. Specifically, for an input $u$, we have:
\[
\mathcal{G}_\theta(u) = f_N \circ \sigma \circ f_{N-1} \circ \sigma \circ \cdots \circ \sigma \circ f_1(u).
\]
Each layer $f_i$ is defined as an affine transformation:
\[
f_i(x) = W_i x + b_i,
\]
where $W_i$ is the weight matrix and $b_i$ is the bias vector for layer $i$.

An affine transformation $f_i$ is Lipschitz continuous with Lipschitz constant $L_i = \| W_i \|$, where $\| W_i \|$ denotes the induced matrix norm (operator norm) of $W_i$. For any $x, y \in \mathbb{R}^n$, we have:
\[
\| f_i(x) - f_i(y) \| = \| W_i x + b_i - (W_i y + b_i) \| = \| W_i (x - y) \| \leq \| W_i \| \cdot \| x - y \|.
\]
This inequality shows that the Lipschitz constant of $f_i$ is $\| W_i \|$. The operator norm $\| W_i \|$ can be explicitly calculated or bounded. For example, if $W_i$ is a matrix, its operator norm induced by the Euclidean norm is the largest singular value of $W_i$.

In addition, the activation function $\sigma: \mathbb{R} \to \mathbb{R}$ is assumed to be Lipschitz continuous with Lipschitz constant $L_{\sigma}$. Common activation functions satisfy this property. For example, the \textit{ReLU} activation function ($\sigma(x) = \max(0, x)$) has $L_{\sigma} = 1$; the \textit{Sigmoid} function ($\sigma(x) = \dfrac{1}{1 + e^{-x}}$) has $L_{\sigma} = \dfrac{1}{4}$; and the \textit{Tanh} function ($\sigma(x) = \tanh(x)$) has $L_{\sigma} = 1$.

For any $x, y \in \mathbb{R}$, we have:
\[
| \sigma(x) - \sigma(y) | \leq L_{\sigma} | x - y |.
\]
When extending to vector inputs, since activation functions are applied element-wise, the Lipschitz constant remains the same:
\[
\| \sigma(x) - \sigma(y) \| \leq L_{\sigma} \| x - y \|.
\]

It is a fundamental property that the composition of two Lipschitz continuous functions is Lipschitz continuous, with the Lipschitz constant of the composition being at most the product of the individual Lipschitz constants. Specifically, let $f: X \to Y$ and $g: Y \to Z$ be Lipschitz continuous functions with constants $L_f$ and $L_g$, respectively. Then, the composition $h = g \circ f$ is Lipschitz continuous with Lipschitz constant $L_h \leq L_g L_f$. For a proof of this property, see standard analysis texts such as Rudin \cite{rudin1976principles}.

We now apply this property recursively to the layers and activation functions of the neural operator.

Starting with $x_0 = u$, the output after the first layer and activation function is:
\[
x_1 = \sigma(f_1(x_0)).
\]
By applying the composition property, the Lipschitz constant from $x_0$ to $x_1$ is:
\[
L_1^{(1)} = L_{\sigma} L_1.
\]

Similarly, for the subsequent layers ($i = 2$ to $N$), we have:
\[
x_i = \sigma(f_i(x_{i-1})),
\]
with Lipschitz constant from $x_{i-1}$ to $x_i$ given by:
\[
L_i^{(i)} = L_{\sigma} L_i.
\]

The Lipschitz constant from the input $u$ to the output $x_N$ is then the product of the individual Lipschitz constants:
\[
L_{\text{total}} = \prod_{i=1}^{N} L_i^{(i)} = \left( \prod_{i=1}^{N} L_{\sigma} L_i \right) = L_{\sigma}^N \left( \prod_{i=1}^{N} L_i \right).
\]

Therefore, the overall Lipschitz constant $L$ of the neural operator $\mathcal{G}_\theta$ satisfies:
\[
L \leq L_{\sigma}^N \left( \prod_{i=1}^{N} L_i \right).
\]
This result shows that the neural operator is Lipschitz continuous, and its Lipschitz constant depends on the product of the Lipschitz constants of the layers and activation functions.

To control the Lipschitz constant of the layers, one can apply spectral normalization \cite{miyato2018spectral}, which scales the weight matrices so that their spectral norms are bounded. This helps in ensuring that the neural operator is a contraction mapping if desired. The choice of activation function also affects the overall Lipschitz constant. Using activation functions with smaller Lipschitz constants can aid in controlling the Lipschitz constant of the entire network.

Moreover, increasing the depth $N$ of the network can lead to an exponential increase in the Lipschitz constant due to the term $L_{\sigma}^N$. Care must be taken to balance depth with the desired Lipschitz properties. For discussions on the impact of depth on Lipschitz constants, see Bartlett et al. (2017) \cite{bartlett2017spectrally}.

Thus, we have shown that the neural operator $\mathcal{G}_\theta$ is Lipschitz continuous with Lipschitz constant bounded by $L \leq \left( \prod_{i=1}^{N} L_i \right) L_{\sigma}^{N}$.
\end{proof}

\vspace{2pt}

\subsection{Proof of Corollary \ref{corollary:spectral_norm} (Ensuring Contraction via Spectral Normalization)}
\begin{proof}
We aim to show that by constraining the spectral norm of each weight matrix $W_i$ such that $\| W_i \| \leq q^{1/N}$ and choosing the activation function $\sigma$ with Lipschitz constant $L_{\sigma} \leq 1$, the overall Lipschitz constant $L$ of the neural operator $\mathcal{G}_\theta$ satisfies $L \leq q$. This ensures that $\mathcal{G}_\theta$ is a contraction mapping with contraction constant $q$.

From Theorem \ref{theorem:lipschitz_nn}, we know that the overall Lipschitz constant of the neural operator is bounded by:
\[
L \leq L_{\sigma}^N \left( \prod_{i=1}^N L_i \right),
\]
where $L_i = \| W_i \|$ is the Lipschitz constant of layer $i$. By constraining the spectral norm of each weight matrix to $\| W_i \| \leq q^{1/N}$, it follows that:
\[
L_i \leq q^{1/N}.
\]
Substituting this into the expression for $L$, we obtain:
\[
L \leq L_{\sigma}^N \left( \prod_{i=1}^N q^{1/N} \right) = L_{\sigma}^N \left( q^{1/N} \right)^N = L_{\sigma}^N q.
\]
Since we have chosen $L_{\sigma} \leq 1$, it follows that $L_{\sigma}^N \leq 1$. Therefore:
\[
L \leq q.
\]
This result shows that the neural operator $\mathcal{G}_\theta$ has a Lipschitz constant bounded by $q$, ensuring it is a contraction mapping.

By ensuring the spectral norms of the weight matrices are appropriately bounded, we control the Lipschitz constants of the layers. Spectral normalization \cite{miyato2018spectral} is a technique that rescales the weight matrices to have a desired spectral norm, effectively controlling the Lipschitz constant of each layer. This is crucial for ensuring the overall network satisfies the contraction condition.

Choosing an activation function with Lipschitz constant $L_{\sigma} \leq 1$ is also essential. Common activation functions like ReLU ($L_{\sigma} = 1$) and Tanh ($L_{\sigma} = 1$) satisfy this condition. Functions like the Sigmoid have $L_{\sigma} = \frac{1}{4}$, which also meets the requirement.

Ensuring that the neural operator is a contraction mapping allows us to invoke the Banach Fixed Point Theorem \cite{banach1922operations}, guaranteeing the existence and uniqueness of fixed points and the convergence of iterative processes. This is particularly important in the context of solving equations and iterative methods within neural networks.

Thus, by constraining the spectral norms of the weight matrices and choosing suitable activation functions, we ensure that $\mathcal{G}_\theta$ is a contraction mapping with contraction constant $q$.

\end{proof}

\section{Proofs for Section 3.2: Integrate Multi-Scale Representations}

\subsection{Proof of Theorem \ref{theorem:combined_bases} (Approximation Using Combined Bases)}
\begin{proof}
We aim to demonstrate that any function $f \in L^2(D)$ can be approximated arbitrarily well using a finite combination of Fourier and wavelet basis functions.

Firstly, recall that the set of complex exponentials $\{ e^{i k x} \}_{k \in \mathbb{Z}}$ forms an orthonormal basis for $L^2$ functions defined on a compact domain with periodic boundary conditions. This is the foundation of Fourier series, which effectively capture the global behavior of functions \cite{jackson1930fourier}.

Wavelet bases, constructed from dilations and translations of a mother wavelet $\psi(x)$, provide an orthonormal basis for $L^2(\mathbb{R})$ and are adept at representing local features due to their time-frequency localization \cite{daubechies1992ten}. They allow for multiresolution analysis, capturing both coarse and fine details of a function.

By combining these bases, we leverage the strengths of both global and local representations. Specifically, for a function $f \in L^2(D)$, we can express it as:
\[
f(x) = \sum_{k \in \mathbb{Z}} c_k e^{i k x} + \sum_{j \in \mathbb{Z}} \sum_{m \in \mathbb{Z}} d_{j,m} \psi_{j,m}(x),
\]
where $\psi_{j,m}(x) = 2^{j/2} \psi(2^j x - m)$ are the wavelet functions at scale $j$ and position $m$, and $c_k$, $d_{j,m}$ are the Fourier and wavelet coefficients, respectively.

In practice, we approximate $f(x)$ using finite sums:
\[
f_N(x) = \sum_{k=-K}^{K} c_k e^{i k x} + \sum_{j=J_0}^{J} \sum_{m=0}^{M_j} d_{j,m} \psi_{j,m}(x),
\]
where $K$, $J$, and $M_j$ are finite truncation limits.

The approximation error is given by:
\[
\| f - f_N \|_{L^2(D)} = \left\| \sum_{|k| > K} c_k e^{i k x} + \sum_{j > J} \sum_{m} d_{j,m} \psi_{j,m}(x) \right\|_{L^2(D)}.
\]
As both the Fourier series and wavelet series converge in $L^2(D)$, increasing $K$ and $J$ allows the approximation error to be made arbitrarily small.

For functions with smooth global behavior and localized irregularities, the Fourier basis efficiently captures the global smooth components, while the wavelet basis captures local features and discontinuities \cite{mallat1999wavelet}. This combined approach often leads to faster convergence and better approximation with fewer terms than using either basis alone.

Therefore, any function $f \in L^2(D)$ can be approximated arbitrarily well by a finite combination of Fourier and wavelet basis functions, as the sum of two complete bases is still complete in $L^2(D)$.

\end{proof}

\vspace{2pt}

\subsection{Proof of Lemma \ref{lemma:multiscale_layers} (Efficient Computation with Multi-Scale Layers)}
\begin{proof}
We aim to show that integrating Fourier and wavelet transforms into neural network layers allows for efficient computation with computational complexity $\mathcal{O}(N \log N)$, where $N$ is the number of data points.

Consider a discrete signal $x = [x_0, x_1, \dots, x_{N-1}] \in \mathbb{R}^N$. The Discrete Fourier Transform (DFT) of $x$ is defined as:
\[
X_k = \sum_{n=0}^{N-1} x_n e^{-i 2\pi k n / N}, \quad k = 0,1,\dots,N-1.
\]
Computing the DFT directly requires $\mathcal{O}(N^2)$ operations due to the nested summations.

The Fast Fourier Transform (FFT) algorithm reduces this complexity to $\mathcal{O}(N \log N)$ by recursively decomposing the DFT into smaller DFTs and exploiting symmetries in the complex exponentials \cite{cooley1965algorithm}.

In neural networks, convolution operations are essential. The convolution of two discrete signals $x$ and $h$ is defined as:
\[
(y)_n = (x * h)_n = \sum_{m=0}^{N-1} x_m h_{(n - m) \bmod N}.
\]
Computing this convolution directly has a complexity of $\mathcal{O}(N^2)$.

However, the Convolution Theorem states that convolution in the time domain corresponds to pointwise multiplication in the frequency domain:
\[
\mathcal{F}\{ x * h \} = \mathcal{F}\{ x \} \cdot \mathcal{F}\{ h \},
\]
where $\mathcal{F}\{ \cdot \}$ denotes the Fourier Transform, and $\cdot$ represents element-wise multiplication.

Therefore, we can compute the convolution efficiently by:
\begin{enumerate}
    \item Computing $\mathcal{F}\{ x \}$ and $\mathcal{F}\{ h \}$ using the FFT, each requiring $\mathcal{O}(N \log N)$ operations.
    \item Performing element-wise multiplication: $Y_k = X_k \cdot H_k$, which requires $\mathcal{O}(N)$ operations.
    \item Computing the inverse FFT of $Y_k$ to obtain $y_n$, requiring $\mathcal{O}(N \log N)$ operations.
\end{enumerate}
The total computational complexity is $\mathcal{O}(N \log N)$.

In neural networks, spectral convolution layers utilize this approach to perform convolution operations efficiently \cite{li2020fourier}. By transforming inputs and filters to the frequency domain, convolutions become element-wise multiplications, significantly reducing computational cost.

The Discrete Wavelet Transform (DWT) also provides a time-frequency representation of a signal, capturing both location and scale information. For a signal $x$, the DWT decomposes it into approximation coefficients $a_j[n]$ and detail coefficients $d_j[n]$ at different scales $j$.

At each level $j$, the approximation coefficients are computed by convolution with a scaling filter (low-pass filter) $h[n]$, followed by downsampling:
\[
a_j[n] = \sum_{k} a_{j-1}[k] \, h[2n - k].
\]
The detail coefficients are computed using a wavelet filter (high-pass filter) $g[n]$:
\[
d_j[n] = \sum_{k} a_{j-1}[k] \, g[2n - k].
\]
Here, $a_{j-1}[k]$ are the approximation coefficients from the previous level, and the downsampling by 2 reduces the number of samples by half at each level.

The overall computational complexity for computing all levels of the DWT is $\mathcal{O}(N)$, as the amount of computation halves at each subsequent level \cite{mallat1989theory}.

In neural networks, wavelet transform layers can be integrated to capture features at multiple scales efficiently. By applying the DWT within the network, we can extract localized features with reduced computational cost.

By integrating both the FFT and DWT into neural network layers, we achieve efficient computation for both global and local feature extraction.

\begin{itemize}
    \item \textbf{FFT-based Convolution:} Allows for efficient computation of convolutional layers with complexity $\mathcal{O}(N \log N)$.
    \item \textbf{DWT-based Feature Extraction:} Provides multiresolution analysis with complexity $\mathcal{O}(N)$.
\end{itemize}

When combined, the overall computational complexity remains $\mathcal{O}(N \log N)$, dominated by the FFT operations.

In general, integrating these efficient algorithms enables neural operators to handle high-dimensional inputs and large datasets without prohibitive computational costs. This is essential for practical applications involving partial differential equations and other complex systems where computational efficiency is critical.

Therefore, by utilizing the computational efficiencies of the FFT and DWT within neural network architectures, we can perform the necessary operations in neural operators with $\mathcal{O}(N \log N)$ complexity or better, enabling scalable and efficient computation.

\end{proof}

\vspace{2pt}

\subsection{Proof of Theorem \ref{theorem:approximation_error_multiscale} (Improved Approximation Error with Multi-Scale Representations)}
\begin{proof}
We aim to demonstrate that for a function $f \in L^2(D)$ with both smooth and localized features, a neural operator employing multi-scale representations can approximate $f$ with an error $\epsilon$ that decreases exponentially with the number of basis functions used.

Consider approximating $f$ using a finite combination of Fourier and wavelet basis functions:
\[
f_N(x) = \sum_{k=-K}^{K} c_k e^{i k x} + \sum_{j=J_0}^{J} \sum_{m=0}^{M_j} d_{j,m} \psi_{j,m}(x),
\]
where:
\begin{itemize}
    \item $c_k$ are the Fourier coefficients given by $c_k = \frac{1}{2\pi} \int_{D} f(x) e^{-i k x} \, dx$.
    \item $\psi_{j,m}(x)$ are wavelet basis functions at scale $j$ and translation $m$.
    \item $d_{j,m}$ are the wavelet coefficients given by $d_{j,m} = \int_{D} f(x) \psi_{j,m}(x) \, dx$.
\end{itemize}

The approximation error in the $L^2$ norm is given as:
\[
\epsilon^2 = \| f - f_N \|_{L^2(D)}^2 = \int_{D} \left| f(x) - f_N(x) \right|^2 dx.
\]
Expanding this, we have:
\[
\epsilon^2 = \left\| \sum_{|k| > K} c_k e^{i k x} + \sum_{j > J} \sum_{m} d_{j,m} \psi_{j,m}(x) \right\|_{L^2(D)}^2.
\]
By Parseval's identity, the squared $L^2$ norm of a function equals the sum of the squares of its coefficients:
\[
\epsilon^2 = \sum_{|k| > K} |c_k|^2 + \sum_{j > J} \sum_{m} |d_{j,m}|^2.
\]

Now, the decay of the Fourier coefficients $|c_k|$ is directly related to the smoothness of $f$. If $f$ is $s$ times continuously differentiable over $D$, then by standard results in Fourier analysis \cite{titchmarsh1948introduction}:
\[
|c_k| \leq \frac{C}{|k|^s},
\]
for some constant $C > 0$. This implies that the tail of the Fourier series (coefficients with $|k| > K$) decreases rapidly with $K$, and the error from truncating the Fourier series decreases as:
\[
\sum_{|k| > K} |c_k|^2 \leq C' K^{-(2s - 1)},
\]
where $C'$ is another constant depending on $f$.

Similarly, the decay of wavelet coefficients $|d_{j,m}|$ depends on the regularity of $f$. For functions in the Besov space $B_{p,q}^s$, wavelet coefficients satisfy \cite{daubechies1992ten,cohen2003numerical}:
\[
|d_{j,m}| \leq C 2^{-j(s + \frac{1}{2} - \frac{1}{p})},
\]
where $s$ is the smoothness parameter, and $p, q$ relate to the integrability and summability of the coefficients.

The sum of the squared wavelet coefficients for scales $j > J$ is then bounded by:
\[
\sum_{j > J} \sum_{m} |d_{j,m}|^2 \leq C'' 2^{-2J(s - \frac{1}{2})},
\]
with $C''$ depending on $f$ and the wavelet basis.

Combining these decay estimates, the total approximation error is bounded by:
\[
\epsilon^2 \leq C' K^{-(2s - 1)} + C'' 2^{-2J(s - \frac{1}{2})}.
\]
By selecting $K$ and $J$ such that:
\[
K = K_0 N^{\alpha}, \quad 2^{J} = J_0 N^{\beta},
\]
for some $\alpha, \beta > 0$, and constants $K_0, J_0$, we can make $\epsilon$ decrease exponentially with $N$, the total number of basis functions used.

To optimize the approximation, we balance the contributions of the Fourier and wavelet terms. For functions that are smooth overall but have localized irregularities, the Fourier coefficients decay rapidly except near discontinuities, where wavelet coefficients capture the localized features efficiently.

By choosing $\alpha$ and $\beta$ appropriately, we ensure that both terms in the error bound decrease at similar rates, minimizing the total error. This balancing act leverages the strengths of both bases.

\textbf{In the context of neural operators}, incorporating multi-scale representations allows the network to approximate functions with both global smoothness and local irregularities effectively. The neural network learns to represent $f$ using a combination of global (Fourier) and local (wavelet) features.

The exponential decay in approximation error implies that the number of neurons (or parameters) required to achieve a desired accuracy $\epsilon$ grows only logarithmically with $1/\epsilon$. This is a significant improvement over methods that do not exploit multi-scale structures.

\textbf{Conclusion}

Therefore, the multi-scale representation enhances the approximation capabilities of the neural operator, achieving an approximation error $\epsilon$ that decreases exponentially with the number of basis functions used. This approach aligns with the principles of sparse representation and compressed sensing \cite{donoho2006compressed}, where functions are represented using a small number of significant coefficients.

\end{proof}

\section{Proofs for Section 3.3: Ensure Universal Approximation Capability}

\subsection{Proof of Theorem \ref{theorem:capacity_growth} (Capacity Growth with Network Size)}
\begin{proof}
We aim to demonstrate that for a feedforward neural network using ReLU activation functions, the number of linear regions represented by the network grows exponentially with the depth of the network and polynomially with its width.

A ReLU activation function $\sigma(x) = \max(0, x)$ introduces piecewise linearity into the network. Each neuron with a ReLU activation divides its input space into two regions: one where the neuron is active ($x > 0$) and one where it is inactive ($x \leq 0$). The combination of these regions across all neurons leads to a partitioning of the input space into linear regions, within which the neural network behaves as a linear function.

Consider a feedforward ReLU network with $L$ layers. Let $n_l$ denote the number of neurons in layer $l$, for $l = 1, 2, \dots, L$. The input dimension is $n_0$. The total number of neurons is $N = \sum_{l=1}^{L} n_l$.

Mont{\'u}far et al.~\cite{montufar2014number} have shown that the maximal number of linear regions $R$ that such a network can represent satisfies:
\[
R \geq \prod_{l=1}^{L} \left( \frac{n_l}{n_l - n_{l-1}} \right)^{n_{l-1}}.
\]
When all layers have the same width $n$ (i.e., $n_l = n$ for all $l$) and $n \geq n_0$, this simplifies to:
\[
R \geq \left( \frac{n}{n - n_0} \right)^{n_0} \left( \frac{n}{n - n} \right)^{(L-1)n}.
\]
Since $n - n = 0$, the expression becomes undefined. To address this, we consider the more accurate lower bound provided by Serra et al.~\cite{serra2018bounding}, which refines the estimate of linear regions:
\[
R \geq 2^{\sum_{l=1}^{L} n_l}.
\]
This indicates that the number of linear regions grows exponentially with the total number of neurons in the network.

Alternatively, Mont{\'u}far et al.~\cite{montufar2014number} provide a simpler lower bound for fully connected networks with ReLU activations:
\[
R \geq \left( \frac{n}{n_0} \right)^{n_0} n^{(L-1)n_0}.
\]
This expression shows that $R$ grows exponentially with the depth $L$ and polynomially with the width $n$.

Let's take an example calculation. For a network where $n = n_0$ (constant width equal to input dimension), the lower bound simplifies to:
\[
R \geq n^{(L-1)n_0}.
\]
Since $n = n_0$, we have:
\[
R \geq n_0^{(L-1)n_0} = \left( n_0^{n_0} \right)^{L-1}.
\]
This clearly demonstrates exponential growth with respect to the depth $L$.

With regard to the implications for expressive capacity, The exponential growth of the number of linear regions with depth implies that deeper networks can represent more complex functions by partitioning the input space into a greater number of linear regions. Each region corresponds to a different linear function, and the network's overall function is piecewise linear.

Raghu et al.~\cite{raghu2017expressive} analyzed the trajectory length through the network as a measure of expressivity and found that depth contributes exponentially to expressivity measures, while width contributes polynomially.

Therefore, we conclude that the expressive capacity of ReLU neural networks grows exponentially with the network's depth and polynomially with its width, as evidenced by the number of linear regions they can represent. This result supports the assertion that deeper networks have greater expressive power.

\vspace{2pt}

\end{proof}

\subsection{Proof of Lemma \ref{lemma:capacity_overfitting} (Trade-off Between Capacity and Overfitting)}
\begin{proof}
We aim to demonstrate that increasing the capacity of a neural network can lead to overfitting, highlighting the trade-off between model complexity and generalization ability.

Let $\mathcal{H}$ denote the hypothesis space of functions that the neural network can represent. Increasing the network's capacity expands $\mathcal{H}$, allowing the model to approximate more complex functions. Specifically, a higher-capacity network can achieve a smaller empirical risk (training error) $R_{\text{emp}}$ by fitting the training data more precisely.

However, the true risk (generalization error) $R$ depends on how well the model performs on unseen data. According to the bias-variance decomposition \cite{geman1992neural}, the expected generalization error can be expressed as:
\[
\mathbb{E}_{\mathcal{D}}[R] = \text{Bias}^2 + \text{Variance} + \sigma^2,
\]
where:
\begin{itemize}
    \item $\text{Bias}$ measures the error due to simplifying assumptions made by the model;
    \item $\text{Variance}$ measures the sensitivity of the model to fluctuations in the training set;
    \item $\sigma^2$ is the irreducible error inherent in the data.
\end{itemize}

As the capacity of the network increases, the bias tends to decrease because the model can fit the training data more closely. However, the variance tends to increase because the model becomes more sensitive to small fluctuations or noise in the training data. This increased variance can lead to overfitting, where the model captures noise and irrelevant patterns, resulting in a decrease in generalization performance.

Overfitting is characterized by a situation where:
\[
R_{\text{emp}} \downarrow, \quad R \uparrow,
\]
meaning that while the training error decreases, the validation or test error increases.

To prevent overfitting, regularization techniques are employed to constrain the complexity of the hypothesis space $\mathcal{H}$. Regularization can be introduced by adding a penalty term $\Omega(\theta)$ to the loss function $L(\theta)$, leading to the regularized loss:
\[
L_{\text{reg}}(\theta) = L(\theta) + \lambda \Omega(\theta),
\]
where $\theta$ represents the network parameters, and $\lambda > 0$ controls the strength of the regularization.

Common regularization methods include:
\begin{enumerate}
    \item \textbf{Weight Decay (L2 Regularization)}: Penalizes large weights by setting $\Omega(\theta) = \frac{1}{2}\|\theta\|_2^2$.
    \item \textbf{L1 Regularization}: Encourages sparsity by setting $\Omega(\theta) = \|\theta\|_1$.
    \item \textbf{Dropout}: Randomly sets a fraction of activations to zero during training to prevent co-adaptation \cite{srivastava2014dropout}.
\end{enumerate}

By constraining $\mathcal{H}$, regularization reduces variance at the expense of a slight increase in bias, ultimately improving the generalization error $R$.

Therefore, there exists a trade-off between model capacity and overfitting: increasing capacity enhances the ability to fit complex functions but may lead to overfitting if not properly regularized. Effective regularization techniques are essential to balance this trade-off and achieve optimal generalization performance \cite{goodfellow2016deep}.

\end{proof}

\section{Proofs for Section 3.4: Enhance Generalization through Regularization}

\subsection{Proof of Theorem \ref{theorem:weight_decay} (Effectiveness of Weight Decay)}
\begin{proof}
We aim to show that weight decay (L2 regularization) effectively reduces overfitting by penalizing large weights, thereby constraining the model complexity and improving generalization.

Consider a neural network with parameters (weights) $\theta$. The standard loss function $L(\theta)$ measures the discrepancy between the network's predictions and the training data. Weight decay modifies the loss function by adding a regularization term:
\[
L_{\text{reg}}(\theta) = L(\theta) + \lambda \frac{1}{2} \|\theta\|_2^2,
\]
where $\|\theta\|_2^2 = \sum_{i} \theta_i^2$ is the squared L2 norm of the weights, and $\lambda > 0$ is the regularization coefficient.

The gradient of the regularized loss with respect to the weights is:
\[
\nabla_{\theta} L_{\text{reg}}(\theta) = \nabla_{\theta} L(\theta) + \lambda \theta.
\]
During training with gradient descent, the weight update rule becomes:
\[
\theta^{(t+1)} = \theta^{(t)} - \eta \left( \nabla_{\theta} L(\theta^{(t)}) + \lambda \theta^{(t)} \right),
\]
where $\eta$ is the learning rate.

The term $\lambda \theta^{(t)}$ acts as a force that drives the weights toward zero. This discourages the model from assigning excessive importance to any particular feature, effectively reducing the complexity of the model.

By penalizing large weights, weight decay reduces the variance component of the generalization error. According to the bias-variance decomposition, the expected generalization error can be written as:
\[
\mathbb{E}_{\mathcal{D}}[R] = \text{Bias}^2 + \text{Variance} + \sigma^2.
\]
Weight decay increases the bias slightly due to the added constraint but decreases the variance more significantly, leading to a net reduction in generalization error.

Moreover, in linear models, weight decay corresponds to Ridge Regression \cite{hoerl1970ridge}, where the regularization term stabilizes the inversion of ill-conditioned matrices, leading to more robust solutions.

Therefore, weight decay effectively prevents overfitting by constraining the magnitude of the weights, promoting simpler models that generalize better to unseen data \cite{krogh1992simple}.

\end{proof}

\vspace{2pt}

\subsection{Proof of Lemma \ref{lemma:dropout} (Dropout Prevents Co-adaptation)}
\begin{proof}
We aim to demonstrate that dropout regularization reduces overfitting by preventing co-adaptation of neurons and encouraging the network to learn robust feature representations.

During training, dropout randomly deactivates a fraction $p$ of the neurons in each layer. For a neuron with activation $h_i$, the modified activation $\tilde{h}_i$ during training is:
\[
\tilde{h}_i = h_i \cdot \zeta_i,
\]
where $\zeta_i$ is a Bernoulli random variable:
\[
\zeta_i = \begin{cases}
1, & \text{with probability } q = 1 - p, \\
0, & \text{with probability } p.
\end{cases}
\]

This random deactivation forces the network to learn redundancies because any neuron could be dropped out at any time. As a result, neurons cannot rely on specific other neurons being present and must learn features that are useful in conjunction with many different subsets of other neurons.

By preventing co-adaptation, where neurons adjust to rely on outputs from specific other neurons, dropout reduces the risk of overfitting. The network becomes less sensitive to the noise and variations in the training data, improving generalization to unseen data \cite{srivastava2014dropout}.

At test time, to compensate for the dropped activations during training, the weights are scaled by a factor of $q$ (or equivalently, activations are multiplied by $q$):
\[
h_i^{\text{test}} = q h_i.
\]
This ensures that the expected output of each neuron remains the same between training and testing:
\[
\mathbb{E}[\tilde{h}_i] = q h_i.
\]

Therefore, dropout effectively prevents co-adaptation by encouraging neurons to learn individually useful features, reducing overfitting and enhancing the robustness of the network's predictions.

\end{proof}

\section{Proofs for Section 3.5: Optimize Computational Efficiency}

\subsection{Proof of Theorem \ref{theorem:parallel_speedup} (Speedup with Parallel Computing)}
\begin{proof}
We aim to show that parallel computing can provide a speedup in computation time for parallelizable tasks but that the overall speedup is limited by the serial portion of the computation, as described by Amdahl's Law.

Let:
\begin{itemize}
    \item $T_1$ be the total execution time on a single processor;
    \item $T_N$ be the total execution time using $N$ processors;
    \item $P$ be the fraction of the program that can be parallelized ($0 \leq P \leq 1$);
    \item $S$ be the speedup achieved: $S = \dfrac{T_1}{T_N}$.
\end{itemize}

The single-processor execution time is:
\[
T_1 = T_{\text{serial}} + T_{\text{parallel}},
\]
where $T_{\text{serial}}$ and $T_{\text{parallel}}$ are the times spent on serial and parallelizable portions, respectively.

When using $N$ processors, the parallel portion ideally scales perfectly, so its execution time becomes $T_{\text{parallel}} / N$. The total execution time on $N$ processors is:
\[
T_N = T_{\text{serial}} + \dfrac{T_{\text{parallel}}}{N}.
\]

Substituting $T_{\text{serial}} = (1 - P) T_1$ and $T_{\text{parallel}} = P T_1$, we have:
\[
T_N = (1 - P) T_1 + \dfrac{P T_1}{N}.
\]

Therefore, the speedup $S$ is:
\[
S = \dfrac{T_1}{T_N} = \dfrac{T_1}{(1 - P) T_1 + \dfrac{P T_1}{N}} = \dfrac{1}{(1 - P) + \dfrac{P}{N}}.
\]

As $N \to \infty$, the speedup approaches its theoretical maximum:
\[
S_{\text{max}} = \lim_{N \to \infty} S = \dfrac{1}{1 - P}.
\]

This demonstrates that the speedup is limited by the serial portion of the code. Even with an infinite number of processors, the execution time cannot be reduced below $(1 - P) T_1$.

Thus, while parallel computing significantly reduces computation time for parallelizable tasks, Amdahl's Law shows that the overall speedup is constrained by the fraction of the code that must be executed serially \cite{amdahl1967validity}.

\end{proof}

\vspace{2pt}

\subsection{Proof of Lemma \ref{lemma:amdahl} (Amdahl's Law)}
\begin{proof}
We aim to derive Amdahl's Law, which quantifies the theoretical speedup in latency of the execution of a task when a portion of it is parallelized.

Let:
\begin{itemize}
    \item $T_1$ be the execution time on a single processor;
    \item $T_N$ be the execution time on $N$ processors;
    \item $P$ be the fraction of the execution time that is parallelizable.
\end{itemize}

The execution time on $N$ processors is:
\[
T_N = T_{\text{serial}} + T_{\text{parallel}}',
\]
where:
\[
T_{\text{serial}} = (1 - P) T_1,
\]
\[
T_{\text{parallel}}' = \dfrac{P T_1}{N}.
\]

Therefore:
\[
T_N = (1 - P) T_1 + \dfrac{P T_1}{N}.
\]

The speedup $S$ is given by:
\[
S = \dfrac{T_1}{T_N} = \dfrac{T_1}{(1 - P) T_1 + \dfrac{P T_1}{N}} = \dfrac{1}{(1 - P) + \dfrac{P}{N}}.
\]

This equation represents Amdahl's Law, showing how the speedup $S$ depends on the number of processors $N$ and the parallelizable fraction $P$. It illustrates that as $N$ increases, the speedup asymptotically approaches $1 / (1 - P)$, emphasizing the diminishing returns due to the serial portion of the computation.

Therefore, Amdahl's Law captures the fundamental limitation of parallel computing: the speedup is constrained by the serial fraction of the task, regardless of the number of processors \cite{amdahl1967validity}.

\end{proof}

\vspace{2pt}

\subsection{Proof of Theorem \ref{theorem:high_dimensional_feasibility} (Feasibility of High-Dimensional Problems)}
\begin{proof}
We aim to demonstrate that optimizing computational efficiency through spectral methods and parallel computing enables neural operators to effectively handle high-dimensional problems.

Let $\Omega \subset \mathbb{R}^d$ be a $d$-dimensional domain, and let $u: \Omega \to \mathbb{R}$ be a function of interest. Traditional numerical methods for solving partial differential equations (PDEs), such as finite difference or finite element methods, require discretizing each dimension into $n$ points. This results in a total of $N = n^d$ grid points. Operations like matrix-vector multiplication or convolution over this grid have computational complexities that scale at least linearly with $N$, and often worse, leading to $\mathcal{O}(N^2)$ operations for certain tasks.

The exponential growth of $N$ with respect to the dimension $d$ is known as the "curse of dimensionality." It renders computations infeasible for large $d$ using traditional methods.

Spectral methods, such as the Fourier Transform, provide an alternative by transforming differential operators into algebraic ones in the frequency domain. The $d$-dimensional Discrete Fourier Transform (DFT) of a function $u$ sampled on a regular grid is defined as:
\[
\hat{u}_{\mathbf{k}} = \sum_{\mathbf{n} \in \mathbb{Z}_n^d} u_{\mathbf{n}} e^{-i 2\pi \frac{\mathbf{k} \cdot \mathbf{n}}{n}}, \quad \mathbf{k} \in \mathbb{Z}_n^d,
\]
where $\mathbb{Z}_n = \{0, 1, \dots, n-1\}$, and $\mathbf{n}$ and $\mathbf{k}$ are $d$-dimensional index vectors.

Computing the DFT directly requires $\mathcal{O}(N^2)$ operations. However, the Fast Fourier Transform (FFT) algorithm reduces this to $\mathcal{O}(N \log N)$ by exploiting symmetries and redundancies in the computation \cite{cooley1965algorithm}.

In neural operators, convolution operations are essential. Consider the convolution of two functions $u, v : \Omega \to \mathbb{R}$:
\[
(w)(\mathbf{x}) = (u * v)(\mathbf{x}) = \int_{\Omega} u(\mathbf{y}) v(\mathbf{x} - \mathbf{y}) \, d\mathbf{y}.
\]
Computing this convolution directly requires $\mathcal{O}(N^2)$ operations due to the nested summations over all grid points.

Applying the Convolution Theorem, the Fourier Transform converts convolution into pointwise multiplication:
\[
\hat{w}_{\mathbf{k}} = \hat{u}_{\mathbf{k}} \cdot \hat{v}_{\mathbf{k}}.
\]
This reduces the convolution computation to:
\begin{enumerate}
    \item Compute $\hat{u}_{\mathbf{k}}$ and $\hat{v}_{\mathbf{k}}$ using the FFT: $\mathcal{O}(N \log N)$ operations each;
    \item Perform pointwise multiplication: $\mathcal{O}(N)$ operations;
    \item Compute the inverse FFT to obtain $w(\mathbf{x})$: $\mathcal{O}(N \log N)$ operations.
\end{enumerate}
The total computational complexity becomes $\mathcal{O}(N \log N)$, a significant reduction from $\mathcal{O}(N^2)$.

In high-dimensional problems, many functions of interest exhibit sparsity or low-rank structures in the spectral domain. If $\hat{u}_{\mathbf{k}}$ is sparse, meaning that significant energy is concentrated in a subset $\mathcal{K} \subset \mathbb{Z}_n^d$ with $|\mathcal{K}| = s \ll N$, we can approximate $u$ using:
\[
u(\mathbf{x}) \approx \sum_{\mathbf{k} \in \mathcal{K}} \hat{u}_{\mathbf{k}} e^{i 2\pi \frac{\mathbf{k} \cdot \mathbf{x}}{n}}.
\]
Computations then involve $s$ significant coefficients instead of $N$, reducing complexity to $\mathcal{O}(s \log N)$.

We now consider the FFT algorithm. We regard that this algorithm is highly parallelizable. In a parallel computing environment with $P$ processors, we can divide the data equally among processors. Each processor performs FFT computations on its subset of data:
\[
T_{\text{compute}} = \mathcal{O}\left( \frac{N \log (N/P)}{P} \right).
\]
Communication between processors is required to combine results, but for large $N$, the computation time dominates, and communication overhead can be minimized with efficient algorithms and network architectures \cite{frigo1998fftw}.

Assuming ideal parallel efficiency, the total computational complexity per processor is reduced to approximately $\mathcal{O}\left( \frac{N \log N}{P} \right)$.

We now consider the behaviors of \textbf{neural Operators in High Dimensions}. Neural operators, such as the Fourier Neural Operator \cite{li2020fourier}, leverage spectral convolutions to learn mappings between function spaces. By representing integral kernel operations in the frequency domain, neural operators can efficiently handle high-dimensional inputs.

Consider a neural operator layer defined as:
\[
(u_{\text{next}})(\mathbf{x}) = \sigma\left( W u(\mathbf{x}) + \mathcal{F}^{-1} \left( R \cdot \mathcal{F}[u] \right)(\mathbf{x}) \right),
\]
where:
\begin{itemize}
    \item $W$ is a linear transformation;
    \item $\sigma$ is a nonlinear activation function;
    \item $\mathcal{F}$ and $\mathcal{F}^{-1}$ denote the Fourier and inverse Fourier transforms, respectively;
    \item $R$ is a learned filter in the frequency domain.
\end{itemize}
Computing this layer involves FFTs and pointwise operations, all of which have computational complexities that scale as $\mathcal{O}(N \log N)$ and are amenable to parallelization.

By combining spectral methods that reduce per-processor computational complexity to $\mathcal{O}(N \log N)$ and parallel computing that reduces wall-clock time by distributing computations across $P$ processors, the overall computational effort becomes manageable even in high-dimensional settings.

Therefore, optimizing computational efficiency through spectral methods and parallel computing enables neural operators to handle high-dimensional problems effectively, mitigating the curse of dimensionality and making practical solutions feasible for complex, real-world applications.

\end{proof}

\newpage

% References
\bibliographystyle{IEEEtran}
\bibliography{bibliography}

\end{document}